\documentclass[a4paper,UKenglish,cleveref, autoref]{article}

\usepackage{multicol}

\usepackage{stmaryrd}
\usepackage{float}
\usepackage{xcolor}
\usepackage{tikz-cd}
\usepackage{amssymb}

\usepackage{comment}
\usepackage{amsthm}
\usepackage{caption}
\usepackage[font={small,it}]{caption}
\usepackage{enumerate}
\usepackage{centernot}
\usepackage{epigraph}
\usepackage[nointegrals]{wasysym}
\usepackage[shortlabels]{enumitem}
\usepackage{tikz}
\usepackage{amsfonts}
\usetikzlibrary{decorations.pathmorphing,shapes}
\usepackage{pgfplots}
\pgfplotsset{width=5cm,compat=1.9}
\pgfplotsset{compat = newest}
\usetikzlibrary{arrows}
\usepackage{stackengine} 
\usepackage{mathtools}

\newcommand{\cqm}{\mathbb W}
\newcommand{\univ}{\mathbb U_\Sigma}
\newcommand{\univnp}{\mathbb U}
\newcommand{\sfont}[1]{\mathfrak{#1}}
\newcommand{\peq}{\preccurlyeq}
\newcommand{\topo}{\text{\sc Top}}
\newcommand{\ctopo}{\text{\sc CTop}}
\newcommand{\sct}{\text{\sc Sct}}
\newcommand{\csct}{\text{\sc CSct}}

\newcommand{\val}[1]{\llbracket #1 \rrbracket}

\newcommand{\ignore}[1]{}

\newcommand{\dn}{\square}

\usepackage[UKenglish]{babel}
\newcommand{\fulllan}{\mathsf{L}^{\blacklozenge\bullet}_{\lozenge}}
\newcommand{\glclan}{\mathsf{L}^{{\bullet}}_{\lozenge}}

\newcommand{\qm}{\mathfrak {Q}}
\newcommand{\ww}{\mathfrak {w}}
\newcommand{\vv}{\mathfrak v}
\newcommand{\Sim}{\mathbf{Sim}}
\newcommand{\sub}{\mathcal {S}}

\newcommand{\VV}{\mathfrak{A}}

\usepackage{multicol}

\usepackage{enumerate}
\newtheorem{theorem}{Theorem}[section]

\newtheorem{corollary}[theorem]{Corollary}
\theoremstyle{definition}\newtheorem{definition}[theorem]{Definition}
\newtheorem{example}{Example}[section]
\newtheorem{lemma}[theorem]{Lemma}

\newtheorem{proposition}[theorem]{Proposition}
\newtheorem{fun-fact}{Fun-Fact}

\usepackage{graphicx}
\usepackage{authblk}

\newcommand{\mleadsto}{\mapstochar\mathrel{\mspace{0.45mu}}\rightsquigarrow}

\begin{document}

\title{Untangled:\\
A Complete Dynamic Topological Logic}

\author[$\mathsection$]{David Fern\'andez-Duque}
\author[$\dagger$]{Yo\`av Montacute}
\affil[$\mathsection$]{Department of Mathematics, Ghent University} 
\affil[ ]{ICS of the Czech Academy of Sciences} 
\affil[ ]{\texttt{david.fernandezduque@ugent.be}}
\affil[$\dagger$]{Computer Laboratory, University of Cambridge}
\affil[ ]{\texttt {yoav.montacute@cl.cam.ac.uk}}

\maketitle

\begin{abstract}
{\em Dynamic topological logic} ($\mathbf{DTL}$) is a trimodal logic designed for reasoning about {\em dynamic topological systems}. 
It was shown by Fern\'andez-Duque that the natural set of axioms for $\mathbf{DTL}$ is incomplete, but he provided a complete axiomatisation in an extended language.
In this paper, we consider dynamic topological logic over {\em scattered spaces}, which are  topological spaces where every nonempty subspace has an isolated point.
Scattered spaces appear in the context of computational logic as they provide semantics for provability and enjoy definable fixed points.
We exhibit the first sound and complete dynamic topological logic in the original trimodal language.
In particular, we show that the version of $\mathbf{DTL}$ based on the class of scattered spaces is finitely axiomatisable over the original language, and that the natural axiomatisation is  sound and complete.
\end{abstract}

\section{Introduction}

In a nutshell, {\em dynamical systems} are mathematical models of movement in space over time.
The interaction between space and time is a fundamental aspect of reality, making such models ubiquitous in many scientific disciplines, ranging from physics to economics.
Computer science is no exception, which should be unsurprising given the temporal aspect of processes and the deep connections between topology and computation, as demonstrated by abstract models of computation such as the well-known Scott domains \cite{Scott82}.

There are many recent examples from pure and applied work in computer science involving dynamical systems.
Lin and Antsaklis~\cite{hybrid} use hybrid dynamical systems in the research of artificial intelligence and computer-aided verification. 
 Brunton and Kutz \cite{brunton2019data} purposed approaching data-related problems through dynamical systems, and Weinan~\cite{weinan2017proposal} suggested modelling nonlinear functions implemented in machine learning using dynamical systems.
Mortveit and Reidys's~\cite{sequential} sequential dynamical systems generalise cellular automata and provide a framework for studying dynamical processes in graphs.
Dynamical systems are also found in their linear form in the shape of Markov chains, linear recurrence sequences and linear differential equations.  
It is therefore not surprising that connections have been established between dynamical systems and algorithms. 
Such links can be found for example in the work of Hanrot, Pujol and Stehl\'e~\cite{hanrot2011analyzing}, and in the work of Chu~\cite{chu2008linear}.
This list is by no means exhaustive.

The applications above warrant the need for an effective formal reasoning framework about {\em topological dynamics}, i.e.~the action of a (typically continuous) function on a topological space.
Modal logic was first suggested to serve that purpose in the 1990s by Artemov et al.~\cite{arte}, who envisioned dynamic topological logic as a bimodal logic for reasoning about topological dynamics.
They defined the logic $\bf S4C$ and showed that it possesses desirable properties such as a natural axiomatisation and the finite model property.
Kremer and Mints~\cite{kmints} suggested that including a third modality, `henceforth' from linear temporal logic ($\mathbf{LTL}$), would lead to a logic powerful enough to reason about the asymptotic behaviour of dynamical systems, possibly leading to applications in automated theorem proving.
They dubbed the resulting system {\em dynamic topological logic} ($\bf DTL$).
They proposed a natural axiomatisation for $\bf DTL$ and conjectured it to be sound and complete for the class of dynamical systems.

However, the situation turned out to be much more intricate than that of $\bf S4C$.
While completeness for Kremer and Mints' calculus has yet to be shown, Fern\'andez-Duque proposed an extension of $\bf DTL$, denoted ${\bf DTL}^*$, which enriches the original language with topological fixed points known as {\em tangled operators}.
He proved that ${\bf DTL}^*$ has a natural axiomatisation~\cite{dtlaxiom}. Later, Fern\'andez-Duque showed that Kremer and Mints' axiomatic system is incomplete; in fact, $\bf DTL$ is not finitely axiomatisable~\cite{FernandezNonFin}.

At this point the status of the problem of axiomatising $\bf DTL$ in the original trimodal language becomes subtle: 
one must search for a non-finite axiomatisation, which is nevertheless `natural' in some sense.
Chopoghloo and Moniri~\cite{Chopo20} proposed an infinitary axiomatisation for $\bf DTL$; the axioms and rules are fairly standard, but the $\omega$-rule, which has infinitely many premises, is allowed.
Thus proofs are infinite objects, unlike the case for ${\bf DTL}^*$, which has infinitely many axioms but finite individual proofs.
To what extent this solves the problem of axiomatising the original trimodal logic is subject to debate.

In this paper we follow a different route and instead restrict our attention to a specific class of dynamical systems, namely, those based on {\em scattered spaces.}
Scattered spaces are topological spaces where every non-empty subspace has an isolated point.
They have gathered attention lately in the context of computational logic, as they may be used to model provability in formal theories~\cite{abashidze1985}, leading to applications in characterising their provably total computable functions \cite{Beklemishev04}. Modal logic on scattered spaces enjoys definable fixed points~\cite{SambinV82}, connecting it to the topological $\mu$-calculus~\cite{BBFMu}.
The latter is particularly relevant to us, as the expressive power gained by topological fixed points, including the tangled operators of $\mathbf{DTL}^*$, is absent in this setting.
As the logic of scattered spaces is the G\"odel-L\"ob modal logic $\bf GL$, we refer to the dynamic topological logic of scattered spaces as {\em dynamic G\"odel-L\"ob logic} ($\bf DGL$).
Moreover, we base our semantics on the Cantor derivative rather than the topological closure, since the former is known to be more expressive \cite{KS}.

Our goal is to demonstrate that the standard finite axiomatisation of $\bf DGL$ is sound and complete, leading to the first complete trimodal dynamic topological logic, as well as the first such logic combining the Cantor derivative with the infinitary `henceforth' from $\bf LTL$.
By the `standard axioms' we refer to the combination of the well-known axiomatisation of $\bf GL$ with $\bf LTL$ axioms for the tenses and $( \newmoon p \wedge \newmoon \Box p) \to\Box\newmoon p $ -- a variant of the continuity axiom of Artemov et al.~adapted for the Cantor derivative.
The proof of completeness employs various advanced techniques from modal logic, including an application of Kruskal's theorem in the spirit of the work of Gabelaia et al.\ \cite{pml}. 

\paragraph{Outline} Section 2 reviews the required definitions and notation necessary to understand the paper. Section 3 focuses on the axiomatisation of the logic and provides an intuitive sketch of the proof of completeness.
Section 4 introduces quasimodels and their corresponding limit models.
Section 5 constructs the universal state space and the simulation formulas of its elements. 
Section 6 assembles the pieces together to derive the completeness proof via  canonical structures. Finally,
Section 7 provides some final remarks and future research directions.

\section{Preliminaries}

Before recalling the definition of dynamic topological logic, let us review some notions from topology and dynamical systems, including the Cantor derivative in a topological space.

\subsection{Topology}

\begin{definition}[topological space]
	A topological space is a pair $\mathfrak X=(X,\tau)$ where $X$ is a set of points and $\tau\subseteq\wp(X)$ is a subset satisfying the following conditions:
	\begin{enumerate}
		\item $X,\varnothing\in\tau$;
		\item if $U,V\in\tau$ then $U\cap V\in\tau$;
		\item if $\mathcal U\subseteq\tau$ then $\bigcup\mathcal U\in\tau$.
	\end{enumerate}
	The elements of $\tau$ are called open sets and $\tau$ is called a \emph{topology} on $X$. Complements of open sets are called \emph{closed sets}.
	If $x\in U\in \tau$, we say that $U$ is a {\em neighbourhood} of $x$.
\end{definition}

We can view partial orders (posets) of the form $(X,\prec)$ as topological spaces with the downset topologies, where each set of the form ${\downarrow x} := \{y:y\peq x$\}, for some $x\in X$, is a basic open set (as usual, $\peq$ denotes the reflexive closure of $\prec$).
Equivalently, a set $U\subseteq X$ is open iff it is downward closed under $\prec$.
Topologies of this form are \emph{Alexandroff topologies,} which have the property that {\em arbitrary} intersections of open sets are open.
Note that in this paper we represent posets via their {\em strict} ordering, i.e., with a transitive, irreflexive relation $\prec$, since it better accommodates our semantics.

Topological spaces can be viewed as an abstract representation of space.
Indeed, the Euclidean spaces $\mathbb R^n$ are the most standard examples of topological spaces.
Here, open sets are all $U\subseteq \mathbb R^n$ for which every $x\in U$ has $\varepsilon>0$ such that $d(x,y)<\varepsilon$ implies $y\in U$, where $d(x,y)$ denotes the Euclidean distance.

A topology on $X$ allows us to define concepts related to limits.
In particular, $x$ is a limit point of $A\subseteq X$ if every neighbourhood of $x$, with respect to the topology on $X$, contains at least one point $a\in A$ distinct from $x$.
This leads to the notion of the Cantor derivative of a subset of $X$.

\begin{definition}[Cantor derivative]
	Let $\mathfrak X=(X,\tau)$ be a topological space. Given $A\subseteq X$, the Cantor derivative of $A$ is the set $d(A)$ of all limit points of $A$. 
	
	Given subsets $A,B\subseteq X$, the Cantor derivative satisfies the following properties:
	\begin{enumerate}
		\item $d(\varnothing)=\varnothing;$
		\item $d(A\cup B)=d(A)\cup d(B)$;
		\item $dd(A)\subseteq A\cup d(A)$.
	\end{enumerate}
\end{definition}

Note that if $X$ is a topological space and $A\subseteq X$, we do not always have that $A\subseteq d(A)$; elements of $A\setminus d(A)$ are called {\em isolated points} of $A$.
Cantor observed that if we iteratively remove isolated points of $X$, we eventually reach the largest subspace $X_\infty\subseteq X$ without isolated points.
The subspace $X_\infty$ may be empty: spaces with this property are known as scattered spaces.
They can be defined without reference to $X_\infty$ as follows:

\begin{definition}[scattered space]
	A topological space $(X,\tau)$ is \emph{scattered} if for every subset $A\subseteq X$
	$$ S\subseteq d(S) \text{ implies } S=\varnothing.$$
	Equivalently, a topological space is scattered if every nonempty subset has an isolated point. 
\end{definition}

Movement in space over discrete time can be modelled by equipping topological spaces with a transition function, which is assumed to be continuous. Recall that if $(X,\tau_X)$ and $(Y,\tau_Y)$ are topological spaces, then $f\colon X\to Y$ is {\em continuous} if whenever $U\subseteq Y$ is open, then $f^{-1}(U) $ is open.

\begin{definition}[dynamic topological system]
	A \emph{dynamic topological system} is a triple $\mathfrak S=(X,\tau,f)$, where $(X,\tau)$ is a topological space and $f:X\rightarrow X$ is a continuous function.
\end{definition}

In this paper, we will mostly be concerned with dynamic topological systems based on a scattered space (or scattered dynamical systems for short).
It is useful to observe that if $(X,\prec)$ is a poset, then $f\colon X\to X$ is continuous iff $x\peq y$ implies $f(x)\peq f(y)$.
The class of all topological spaces will be denoted by $\topo$ and the class of all dynamical systems by $\ctopo$. In addition, the class of scattered spaces will be denoted by $\sct$ and the class of all scattered dynamical systems by $\csct$.
Our goal is to axiomatise the dynamic topological logic of the systems in $\csct$, as defined in the following subsection.

\subsection{Dynamic topological logic}
We introduce the language with which we will be working with throughout the paper. Given a nonempty set $\mathsf{PV}$ of propositional variables, the language of the logic $\mathbf{DGL}$ is defined recursively as follows:
$$ \varphi::=p\;|\;\varphi\wedge\varphi\;|\;\neg\varphi\;|\;\lozenge\varphi\;|\;\newmoon\varphi\;|\;\blacklozenge\varphi,$$
where $p\in\mathsf{PV}$. It consists of the Boolean connectives $\wedge$ and $\neg$, the temporal modalities `next' $\newmoon$ and `eventually' $\blacklozenge$ with its dual `henceforth' $\blacksquare:=\neg\blacklozenge\neg$, and the spatial modality $\lozenge$ for the Cantor derivative with its dual the co-derivative $\square:=\neg\lozenge\neg$. We define other connectives (e.g.\; $\vee$, $\rightarrow$) in the usual way. 

This language will be denoted from this point onward by $\fulllan$ while the language without the henceforth operator, the language of the logic $\mathbf{GLC}$ (G\"odel-L\"ob logic with Continuity), will be denoted by $\glclan$.

\begin{definition}[semantics]
	A \emph{dynamic topological model} is a tuple $\sfont M=(X,\tau,f,\nu)$, where $(X,\tau,f)$ is a dynamic topological system and $\nu:\mathsf{PV}\rightarrow\wp(X)$ is a {\em valuation function}. 
	Given $\varphi\in\fulllan$, we define the \emph{truth set} $\val\varphi \subseteq X$ of a formula $\varphi$ as follows:
	\setlength\columnsep{-3mm}
	\begin{multicols}{2}
\noindent
	\begin{itemize}
		\item $\val p=\nu(p)$;
		\item $\val{\neg\varphi} = X\backslash \val\varphi $;
		\item $\val{\varphi\wedge\psi}=\val\varphi \cap \val \psi $;
		\item $\val{\lozenge\varphi}=d(\val\varphi)$;
		\item $\val{\newmoon\varphi} =f^{-1}(\val\varphi )$;
		\item $\val{\blacklozenge\varphi} =\bigcup_{n\geq 0}f^{-n}(\val\varphi)$.
	\end{itemize}	
	\end{multicols}
We write $\sfont M,x\models\varphi $ if $x\in \val\varphi$ and $\sfont M\models \varphi$ if $ \val\varphi =X$. We may also denote a specific truth assignment by $\val\cdot_\sfont M$ or $\val\cdot_\mathfrak \nu$ if we deal with more than one possible model or valuation. 
\end{definition}

\section{Axiomatisation}
It was shown by Esakia \cite{esakia} and Simmons \cite{simmons} that the logic $\mathbf{GL}$, whose characteristic axiom is $  \square(\square \varphi \rightarrow \varphi) \rightarrow \square \varphi  $, is the logic of all scattered spaces with respect to the topological semantics where $\lozenge$ is interpreted as the Cantor derivative operation. 
Aside from this change and a modified continuity axiom, our axiomatisation of $\mathbf{DGL}$ is very similar to Kremer and Mints' axiomatisation~\cite{kmints} and consists of the following axiom schemes:
\begin{itemize}
	\item{\rm Taut} $:=\text{All propositional tautologies}$
	\item{\rm K} $:= \dn(\varphi\to\psi)\to(\dn\varphi\to \dn\psi)$
	\item{\rm L} $:=    \square(\square \varphi \rightarrow \varphi) \rightarrow \square \varphi  $
	\item{${\rm Next}_\neg$} $:=\neg\newmoon\varphi\leftrightarrow\newmoon\neg\varphi$
\item{${\rm Next}_\wedge$} $:=\newmoon (\varphi\wedge\psi)\leftrightarrow \newmoon \varphi \wedge\newmoon \psi $
\item{$\rm C$} $:= (\newmoon\varphi\wedge \newmoon\dn\varphi)\to  \dn\newmoon\varphi$
\item{${\rm K}_\blacksquare$} $:=\blacksquare(\varphi\to\psi)\to (\blacksquare\varphi\to\blacksquare\psi)$
\item{${\rm Fix}_\blacksquare$} $:=\blacksquare\varphi\rightarrow(\varphi\wedge\newmoon\blacksquare\varphi)$
\item{${\rm Ind}_\blacksquare$} $:=\blacksquare(\varphi\rightarrow\newmoon\varphi)\rightarrow(\varphi\rightarrow\blacksquare\varphi)$
\end{itemize}
It also has the following inference rules:
\begin{multicols}{2}
	\begin{itemize}
\item{\rm MP} $:= \dfrac{\varphi \ \ \varphi\to \psi}\psi$
\item{${\rm Nec}_\dn$} $:= \dfrac{\varphi }{\dn \varphi}$
\item{${\rm Nec}_{\newmoon}$} $:= \dfrac{\varphi}{\newmoon \varphi}$
\item{${\rm Nec}_\blacksquare$} $:= \dfrac{\varphi}{\blacksquare \varphi}$
\end{itemize}
\end{multicols}
We write $\mathbf{DGL}\vdash\varphi$ or simply $\vdash\varphi$ if $\varphi$ is derivable using these rules and axioms.

Given a dynamic topological system $\mathfrak S=(X,\tau,f)$, the intuition behind the axioms above can be stated briefly as follows:
 the axiom $\rm L$ expresses  transitivity and  well-foundedness \cite{segerberg1971}, and in the case of a topology $\tau$, it expresses that $\tau$ is a scattered space~\cite{esakia}.
\begin{lemma}
A topological space $\mathfrak X=(X,\tau)$ is scattered if and only if $\mathfrak X\models \rm L$ is sound for the class of scattered spaces.
\end{lemma}

\begin{proof}
We will use the contrapositive of $\rm L$, $\lozenge p\to \lozenge(p\wedge\Box\neg p)$, for convenience. 
\medskip

\noindent $(\Rightarrow)$ Suppose that $\mathfrak X$ is scattered and fix a valuation $\nu$ on $X$.
Suppose that $\mathfrak X,x\models \lozenge p$. Then in each open set $U$ of $x$ there exists a point $y\neq x$ such that $\mathfrak X,y\models p$. 
Clearly, for every valuation $\nu$ on $\mathfrak X$ and an open set $U$ of $x$, $S=U\cap\val p_\nu \neq\varnothing$. 
Since $\mathfrak X$ is scattered, $S$ contains a point $y$ isolated in $S$. 
In particular, there exists an open set $V$ of $y$ for which $V \cap S=\{y\}$. 
Since $y\in\val p_\nu$ and $V\setminus{y}\cap\val p_\nu=\varnothing$, $\mathfrak X,x\models \lozenge (\square\neg p\wedge p)$, as required.
\medskip

\noindent
$(\Leftarrow)$ Suppose that $\mathfrak X=(X,\tau)$ is not scattered. 
Then there exists a non-empty subset $S\subseteq X$ without an isolated point in $S$. 
We define a valuation $\nu$ on $\mathfrak X$ such that $\val p_\nu=S$.
Note that for every $y\in S$ and every neighbourhood $U$ of $y$ we have $\{y\}\subsetneq S\cap U$, hence $\mathfrak X,y\models \lozenge p$.
Since $S\neq\varnothing$, there is at least one $y_0\in X$ such that $\mathfrak X,y_0\models \lozenge p$.
In addition $X\setminus S=\val{\neg p}_\nu$.
Since $X$ is open, $\mathfrak X,y_0\models \square(\lozenge p\vee \neg p)$, witnessing that the $\rm L$ axiom fails on $y_0$.
\end{proof} 

The two operators $\rm Next_\neg$ and $\rm Next_\wedge$ express the functionality of the map $f:X\rightarrow X$, and the axiom $\rm C$ expresses that $f$ is continuous. 
Finally, the two axioms {${\rm Fix}_\blacksquare$} and {${\rm Ind}_\blacksquare$} express the properties of fixed-point and successor induction of $\blacksquare$, which dictate the behaviour of the `henceforth' operation. 
Each of these axioms is proven sound in either \cite{kmints} or \cite{D-FM21}, yielding the following:

\begin{proposition}[soundness]
	The axiomatisation above is sound for the class of scattered dynamical systems. 
\end{proposition}

 The logic $\mathbf{GLC}$ is the same as $\mathbf{DGL}$, but as its language lacks the `henceforth' operator the corresponding axioms are omitted.
When a formula $\varphi$ is derivable in $\mathbf{GLC}$ we may write $\mathbf{GLC}\vdash\varphi$, although as mentioned $\vdash$ without a specified logic refers to derivability in $\mathbf{DGL}$.
 Nevertheless, our proof of completeness will use the following result by Fern\'andez-Duque and Montacute \cite{D-FM21}.
 
\begin{theorem}[$\mathbf{GLC}$ completeness]\label{glcomp}
The logic $\mathbf{GLC}$ is complete and has the finite model property with respect to the class of scattered dynamical systems.
\end{theorem}

\noindent In particular, note that every validity in $\mathbf{GLC}$ is syntactically derivable. 
This will become very useful in our proof of completeness for $\mathbf{DGL}$, which can be stated as follows:




\begin{theorem}[completeness]\label{dglcomp}
	$\csct\models\varphi$ implies $\vdash\varphi$, i.e\ all formulas valid on the class of scattered dynamical system are syntactically derivable in $\mathbf{DGL}$.
\end{theorem}
The rest of the paper is devoted to this result.
It involves several elements, so it will be useful to sketch their role in the proof.

The general idea is to adapt a proof of completeness for linear temporal logic (see e.g.~\cite{temporal}).
For readers familiar with completeness proofs of $\mathbf{LTL}$, we recall two standard approaches.
The first is to construct the (infinite) canonical model and then perform filtration to obtain a finite model.
Filtration is needed since in the canonical model the relation used for interpreting $\blacklozenge$ is not necessarily the real transitive, reflexive closure of the successor function.
While this property {\em does} hold in the filtrated model (obtained by taking a suitable quotient), the drawback is that after filtration, the modality $\newmoon$ is no longer interpreted via a function.
We remedy this by `unwinding'; that is, choosing a path $[w_0],[w_1],\ldots$ of elements of the filtrated model, where $[w]$ denotes the equivalence class of $w$.
In the terminology of the present paper, such a path is a {\em realising path.}

This approach does not work in our setting since filtration destroys the continuity condition (which {\em does} hold in the canonical model of $\mathbf{DGL}$).
Instead, we follow something closer to the second approach, where we begin with a structure that looks like the final filtrated model, but might include `too many' points.
To this end, fix a finite set $\Sigma$ closed under subformulas and single negations (typically, the subformulas of some `target formula' $\varphi$).
In the $\mathbf{LTL}$ setting, a `point' of this model would be a {\em type} $\Theta$, i.e.~a subset of $\Sigma$ respecting Booleans: in particular, $\neg \psi \in \Theta$ iff $ \psi\notin \Theta$, for $\psi\in\Sigma$.
Other conditions may be imposed on types, e.g.~$\blacksquare \psi\in \Theta $ implies $\psi\in \Theta$.
Let $T_\Sigma$ denote the set of all $\Sigma$-types.
Using the truth conditions of the tenses $\newmoon$ and $\blacklozenge$, we may define a `successor relation' $S_\Sigma$ on the set of $\Sigma$-types, so that for example if $\Theta \mathrel S_\Sigma \Delta$ and $\newmoon\psi\in \Theta$, then $\psi\in \Delta$.
As was the case with the canonical model,  $\blacklozenge\psi\in \Theta$ does not necessarily imply that there is $n$ and $\Delta$ such that $\Theta\mathrel S^n_\Sigma \Delta$ and $\psi\in \Delta$.
But in this case, rather than a quotient, we should take a {\em subset} of $T_\Sigma$. 
Say that a type $\Theta$ is {\em consistent} if $\chi(\Theta) := \bigwedge \Theta$ is consistent with respect to the axioms and rules of $\mathbf{LTL}$.
Let $\cqm_\Sigma$ be the restriction of $(T_\Sigma,S_\Sigma)$ to the set of consistent types.
Then, much as was the case for the filtrated canonical model, $\cqm_\Sigma$ {\em does} interpret $\blacklozenge$ correctly, but $S_\Sigma$ is not functional.
As before, we obtain a proper $\mathbf{LTL}$ model by choosing a realising path on $\cqm_\Sigma$.

Our proof of completeness of $\mathbf{DGL} $ {\em grosso modo} follows this second proof sketch.
The biggest change is that types must be replaced by more complex objects.
Conceptually, we may think of types as describing the state of affairs (relative to $\Sigma$) at a given moment in time.
However, in the setting of dynamical systems, this involves not only stating which propositions hold, but also describing the `local' topological structure.
As the purely topological fragment of $\mathbf{DGL}$ is just $\mathbf{GL}$ (i.e., the logic of scattered spaces), and $\mathbf{GL}$ is sound and complete for finite (strict) posets, we will let $\Sigma$-states be finite posets labelled by types: 
formally, a {\em $\Sigma$-state} is a structure $\mathfrak w=(|\mathfrak w|,\prec_\mathfrak w,\ell_\mathfrak w,0_\mathfrak w)$, where  $(|\mathfrak w|,\prec_\mathfrak w)$ is a strict, finite poset with a root $0_\mathfrak w$, and $\ell_\mathfrak w$ assigns a $\Sigma$-type $\ell_\mathfrak w(w)$ to each $w\in |\mathfrak w|$, satisfying some constraints to mimic the semantics of $\lozenge$.

The set of all $\Sigma$-states forms a structure which we denote $\univ$, and plays the role of $(T_\Sigma, S_\Sigma)$ in the $\mathbf{LTL}$ completeness proof.
This structure will be defined in Section \ref{secSimStat}.
As was the case in the $\mathbf{LTL}$ proof, $\univ$ contains `too many' points, and so we must eliminate those $\Sigma$-states that are `inconsistent'.
This involves describing a $\Sigma$-state $\mathfrak w$ within our formal language.
It is well known that finite frames can be described up to bisimulation in the modal language, but as it turns out, we need to describe states up to {\em simulation} (rather than bisimulation).
Simulation formulas are built in Section \ref{secSimFor}.
The simulation formula for $\mathfrak w$ is denoted $\Sim(\mathfrak w) $ and plays the role of $\chi(\Theta)$ in the $\mathbf{LTL}$ completeness proof.

With this, in Section \ref{canquas} we define $\cqm_\Sigma$, the restriction  of $\univ$ to the set of consistent $\Sigma$-states, i.e.~those $\Sigma$-states $\mathfrak w$ such that $\Sim(\mathfrak w)$ is consistent with our axiomatisation of $\mathbf{DGL}$.
The structure $\cqm_\Sigma$ {\em does} satisfy the required properties to ensure satisfiability.
To be precise, $\cqm_\Sigma$ is a {\em quasimodel,} a labelled strict poset which, aside from having a non-deterministic transition relation rather than a function, respects all semantic clauses of $\fulllan$.
Quasimodels are quite general, with $\cqm_\Sigma$ being only a special case, and as such they are defined much earlier, in Section \ref{secQuasi}.
As it was in the case of $\mathbf{LTL}$, a proper model may be obtained by extracting realising paths from any quasimodel $\qm$.
The major difference in our setting is that now we must simultaneously consider {\em all} realising paths, as they form a dynamic topological system, called the {\em limit model} of $\qm$ and denoted $\vec{\qm}$.
By defining the topology of $\vec{\qm}$ in the right way, we in fact obtain a scattered dynamical system satisfying all formulas that were already satisfied in $\qm$.

As a final remark, note that contrary to the $\mathbf{LTL}$ setting, the structure $\cqm_\Sigma$ for $\mathbf{DGL}$ is not finite.
As we often need to consider disjunctions or conjunctions of formulas of the form $\Sim(\mathfrak w)$, and formulas are finite objects, this is a delicate issue when adapting the $\mathbf{LTL}$ proof.
Fortunately, at each point in the proof, we may restrict our attention to finite sets of $\Sigma$-states: this is a deep fact that relies on an application of Kruskal's theorem pioneered by Gabelaia et al.~\cite{pml}.
This will come into play in Section \ref{sectCanQuasSig}, where we show that $\cqm_\Sigma$ indeed respects the semantics of $\blacklozenge$.

As the treatment of quasimodels and their associated limit models does not depend on the construction of $\univ$ and $\cqm_\Sigma$, we postpone it until later in the paper and first focus our attention on a general treatment of quasimodels.

\section{Quasimodels and limit models}\label{limitmodelsfrom}



In this section, we introduce quasimodels, which are similar to scattered dynamical systems based on an Aleksandroff space (represented as the downset topology induced by a strict partial order). 
The only difference is that the transition function of quasimodels is replaced with a non-deterministic relation.
These structures will be useful in our completeness proof, as quasimodels are easier to construct than proper models.
In order to maintain the validity of expressions such as $\newmoon (p\vee q)\leftrightarrow \newmoon p\vee\newmoon q $, we equip each quasimodel $\qm$ with a labelling function $\ell_\qm$ that assigns a {\em type} to each point.
In the main result of this section we show that for every quasimodel $\qm $, the limit model $\vec{\qm} $ is a dynamic topological model satisfying every formula satisfied by $\qm$, i.e.~those formulas in the range of $\ell_\qm$.

\subsection{Quasimodels}\label{secQuasi}
Given a formula $\varphi$, we denote by $\sub(\varphi)$ the set of subformulas of $\varphi$, and we define $\sub_\pm(\varphi)=\sub(\varphi)\cup\{\neg\psi: \psi\in\sub(\varphi)\}$.

\begin{definition}[type]\label{sigtype}
A set $\Phi\subseteq \fulllan $ is a \emph{type} if the following conditions are satisfied:
\begin{enumerate}

	\item There is no formula $\varphi$ such that $\varphi\in \Phi$ and $\neg\varphi\in\Phi$;
	
	\item if $\neg\neg\varphi\in \Phi$ then $\varphi\in \Phi$;

	\item if $\varphi\wedge\psi\in\Phi$ then $\varphi,\psi\in\Phi$;
	
	\item if $\neg(\varphi\wedge\psi)\in\Phi$ then $\neg\varphi\in \Phi$ or $\neg\psi\in\Phi$;

	\item if $\blacksquare\varphi\in\Phi$ then $\varphi\in\Phi$.
\end{enumerate}
The set of all types is denoted by $\mathbb T$.
If $\Sigma$ is a set of formulas closed under subformulas and single negations, we say that $\Phi$ is a {\em $\Sigma$-type} if $\Phi\subseteq \Sigma $ and, for every $\varphi\in\Sigma$, either $\varphi\in \Phi$ or $\neg\varphi\in \Phi$ (identifying $\varphi$ with its double negation as needed).
We denote by $\mathbb{T}_\Sigma$ the set of all $\Sigma$-types.
Often we will assume that $\Sigma$ is finite: when $\Sigma\subseteq\fulllan$ is finite and closed under subformulas and single negations, we write $\Sigma\Subset \fulllan$.
\end{definition}


\begin{definition}[labelled poset] A {\em labelled poset} is a triple $\sfont A=(|\sfont A|,\prec _\sfont A,\ell_\sfont A)$, where $|\sfont A|$ is a set of points, $\prec_\sfont A$ is a strict partial order on $|\sfont A|$, and $\ell_\sfont A:|\mathfrak{A}|\rightarrow \mathbb T_\Sigma$ is a labelling function such that
\begin{itemize}

\item $ \lozenge \varphi\in \ell_{\sfont A}(w)$ implies $\exists v(v\prec w \mathrel\& \varphi\in \ell_{\sfont A}(v)).$

\item $ \square \varphi\in \ell_{\sfont A}(w)$ implies $\forall v(v\prec w \Rightarrow\varphi\in \ell_{\sfont A}(v)).$

\end{itemize}
If $\Sigma$ is a set of formulas and the range of $\ell_\sfont A$ is contained in $\mathbb T_\Sigma$, we say that $\sfont A$ is a {\em$\Sigma$-labelled poset.}
\end{definition}

For our purposes, a continuous relation on a topological space is a relation for which the preimage of every open set is open.
In the context of posets, a relation $S\subseteq|\sfont A|\times|\sfont B|$ between orders $\sfont A$ and $\sfont B$ is called \emph{continuous} if it satisfies the forward-confluence property, i.e.\ if $w\peq_{\sfont A} w'$ and $wSv$, then there is $v'$ such that $w'Sv'$ and $v\peq_{\sfont B} v'$. 
This corresponds to the topological notion of continuity with respect to the downset topologies induced by $\prec_\sfont A$ and $\prec_\sfont B$. 
\begin{definition}[sensibility] Suppose that $\Phi,\Psi\in\mathbb{T} $. 
The ordered pair $(\Phi,\Psi)$ is \emph{sensible} if
\begin{enumerate}
	\item $\newmoon\varphi\in\Phi$ implies that $\varphi\in\Psi$ and $\neg\newmoon\varphi\in\Phi$ implies that $\neg\varphi\in \Psi$;
	\item $\blacklozenge\varphi\in\Phi$ implies that $\varphi\in\Phi$ or $\blacklozenge\varphi\in\Psi$;
	\item $\blacksquare\varphi\in\Phi$ implies that $\blacksquare\varphi\in\Psi$.
\end{enumerate}
	Accordingly, a pair of points $(w,v)$ in a labelled poset $\VV$ is \emph{sensible} if $(\ell(w),\ell(v))$ is sensible. 
	A continuous relation $S\subseteq |\VV|\times|\VV|$ is sensible if every pair in $S$ is sensible. Moreover, $S$ is $\omega$-\emph{sensible} if it is serial and whenever $\blacklozenge\varphi\in\ell(w)$, there is $n\geq 0$ and there is a point $v$ such that $w S^n v$ and $\varphi\in\ell(v)$.
	\end{definition}
	We now have everything we need in order to provide the definition of a \emph{quasimodel}.
	Below, a poset $(W,\prec)$ is {\em locally finite} if ${\downarrow }w$ is finite for all $w\in W$.

\begin{definition}[quasimodel] 
	A \emph{weak quasimodel} is a tuple $\qm =(|\qm|,\prec_\qm,\ell_\qm,S_\qm)$, where  $(|\qm|,\prec_\qm,\ell_\qm)$ is a locally finite labelled poset and $S_\qm\subseteq |\qm|\times |\qm|$ is a sensible relation.
	 If in addition $S_\qm$ is $\omega$-sensible, then $\qm$ is said to be a {\em quasimodel,} and if the range of $\ell_\qm$ is contained in $\mathbb T_\Sigma$, we say that $\qm$ is a {$\Sigma$-quasimodel} (or {\em weak $\Sigma$-quasimodel,} if $\qm$ is not $\omega$-sensible).
\end{definition}

	 We adopt the general convention that subscripts in e.g.~$\prec_\qm$ or $\ell_\qm$ will be dropped when this does not lead to confusion. Nevertheless, the subscripts will come in handy when multiple structures are involved.
	 
	 \begin{example} Let $\varphi=\blacksquare(\square p\wedge p) \rightarrow\square\blacksquare p$.
	 The following structure is a quasimodel, under the labelling given by $\ell_\qm(w)=\{\neg p,\ldots\}$, $\ell_\qm(v)=\{\square p\wedge p,\neg\blacksquare p,\ldots\}$ and
	 $\ell_\qm(u) = \{\blacksquare(\square p\wedge p),\neg \square \blacksquare p, \neg\varphi,\ldots\}$ (where `$\ldots$' indicates formulas omitted for simplicity).
	 
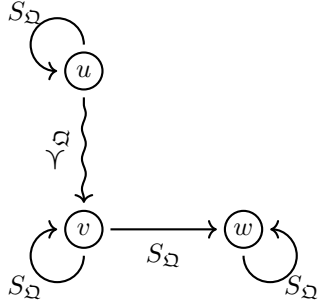
\begin{figure}[H]
	 \centering
	 \begin{tikzpicture}[scale=0.7,decoration={snake, 
    amplitude = .25mm,
    segment length = 4mm,
    pre length=0.6mm}]
\draw[thick] (0,0) circle (.35);
\draw (0,0) node {$u$};
 
\draw[thick,->] (0,+.5) arc (0:+270:.5) ;
 
\draw (-1.1,+1.1) node {$S_\qm$};

\draw[thick] (0,-3) circle (.35);
\draw (0,-3) node {$v$};
\draw[thick,->] (0,-3.5) arc (0:-270:.5) ;
 
\draw (-1.1,-4.1) node {$S_\qm$};
\draw[thick,-> ] (.5,-3) -- (2.5,-3);
\draw (1.5,-3.5) node {$S_\qm$};
\draw[thick,<-,decorate] (0,-2.5) -- (0,-.5);
\draw (-.5,-1.5) node {\large$\rotatebox[origin=c]{90}{$\prec_\qm$}$};
\draw[thick] (3,-3) circle (.35);
\draw (3,-3) node {$w$};
\draw[thick,->] (3,-3.5) arc (-180:+90:.5) ;
 
\draw (4.1,-4.1) node {$S_\qm$};
\end{tikzpicture}
\end{figure}
This quasimodel falsifies $\varphi$, but it is known that the formula $\varphi$ is valid on every Aleksandroff topological space \cite{kmints}. We will see that the quasimodel above witnesses that $\varphi$ is not a theorem of $\mathbf{GLC}$. 
	 \end{example}
	 \begin{example}\label{examModisQ}
	 Quasimodels generalise dynamic poset models (i.e. dynamic topological models with the downset topology) in the following sense:
	 Suppose that $\sfont M$ is such a model, and let $\Sigma$ be any set of formulas closed under subformulas.
	 For $w\in |\sfont M|$, let $\ell_\Sigma (w) = \{\varphi\in \Sigma:w\in \val\varphi_\sfont M\}$.
	 Then, it is not hard to check that $(|\sfont M|,\prec_\sfont M,S_\sfont M,\ell_\Sigma)$ is indeed a (deterministic) $\Sigma$-quasimodel.
	 Henceforth, we will tacitly identify dynamic poset models with their associated quasimodel.
	 \end{example}

\subsection{Limit models}\label{secLimMod}
Once the notion of a quasimodel has been defined, we need to associate to each quasimodel a corresponding limit model.
We will construct it one part at a time, beginning with defining the elements of the domain.

\begin{definition}[realising path] Let $ \qm=(|\qm|,\prec,S,\ell)$ be a $\Sigma$-quasimodel. 
A path in  $\qm$ is a sequence $(w_n)_{n<\alpha}$ with $\alpha\leq\omega$ such that $w_mSw_{m+1}$, where $m+1<\alpha $. An infinite path $\vec{w}=(w_n)_{n<\omega}$ is called a \emph{realising path} if for all $m<\omega$ and $\blacklozenge\varphi\in\ell(w_m)$, there exists $k\geq m$ such that $\varphi\in \ell(w_k)$.  
\end{definition}
We denote the set of realising paths on $\qm$ by $|\vec{\qm}|$.
This set will be used as the universe of the limit model $\vec{\qm}$. 
We use the shift operator $\sigma$, defined by $\sigma((w_n)_{n<\omega})=(w_{n+1})_{n<\omega}$, as the transition function on $|\vec{\qm}|$.

\begin{lemma}\label{extensive}
Let $\qm=(|\qm|,\prec,S,\ell)$ be a $\Sigma$-quasimodel. Then
	\begin{enumerate}
		\item $|\vec \qm|$ is closed under $\sigma$;
		\item any finite path $(w_1,\dots, w_n)$ in $\qm$ can be extended to an infinite realising path $\vec{w}=(w_i)_{i<\omega}\in |\vec\qm|$.
	\end{enumerate}
\end{lemma} 
\begin{proof}
	A proof for this statement can be found in \cite{FernandezNonDeterministic}.
\end{proof}
The following lemma follows from the forward-confluence of $S$ by a straightforward induction on $n$:
\begin{lemma}\label{paths}
	Let $\qm$ be a $\Sigma$-quasimodel and let $(w_i)_{i\leq n}$ be a finite path. Let $v_0$ be such that $v_0\preccurlyeq w_0$.  Then, there exists a path $(v_i)_{i\leq n}$ such that $v_i\preccurlyeq w_i$, for $i\leq n$.
\end{lemma}

Fern\'andez-Duque~\cite{FernandezNonDeterministic} showed that $|\vec{\qm}|$ can be equipped with a natural topology making the shift operator continuous.
However, the topology given there is not necessarily scattered, even if $\prec_\qm$ is well-founded.
Fortunately, the construction can be modified to ensure that the resulting space is indeed scattered.

\begin{proposition}
Let $\vec{w}$ be a realising path. We define the $m$-neighbourhood $N_m({\vec{w}})$ of $\vec{w}$ to be the set of all realising paths $\vec{v}$ such that
\begin{enumerate}[(i)]
	\item $ v_i \preccurlyeq w_i $ for all $ i<m$;
	\item\label{caseEqual} if $v_k = w_k$ for some $k<m$, then $v_j = w_j$ for all $j\geq k$.
\end{enumerate}
Then, the collection $\mathcal B_\prec$ of such neighbourhoods forms a topological basis on $|\vec{\qm}|$.
\end{proposition}
\begin{proof}
    First, note that $\bigcup_{B\in\sfont B_\prec}B=|\vec\qm|$. This is the case since for any realising path $\vec{w}\in |\vec{\qm}|$, $N_0(\vec{w})=|\vec\qm|$ as (i) and (ii) vacuously hold for every $\vec{v}\in|\vec\qm|$.
    
    Next, we prove that for all $B_1,B_2\in \sfont B_\prec$ and $\vec{u}\in B_1\cap B_2$, there exists $B_3\subseteq B_1\cap B_2$ such that $\vec{u}\in B_3$. Let $\vec{u}\in N_i(\vec{w})\cap N_j(\vec{v})$ and suppose without loss of generality that $i\leq j$. We prove that $N_j(\vec u)\subseteq N_i(\vec{w})\cap N_j(\vec{v}) $. 
    
    First, note that for all $r<j$ we have that $u_r\preccurlyeq v_r$ by definition, but also $u_r\preccurlyeq w_r$ since either $u_r\prec w_r$ for all $r<j$ or there exists some $k<m$ for which $u_k=w_k$ and then by (ii) this holds from $k$ henceforth.
    
     Let $\vec x\in N_j(\vec u)$. 
     By definition, $x_r\preccurlyeq u_r$ for all $r<j$. 
     By transitivity of $\prec$ we have two cases to consider: either $x_r\preccurlyeq w_r$ for all $r<i$ and then by (i) we have that $\vec{x}\in N_i(\vec w)$; 
     or $x_k=w_k$ for some $k<i$, but then $w_k= u_k$ since $x_k\preccurlyeq u_k\preccurlyeq w_k$ and $\prec$ is irreflexive. In particular, $w_r=u_r$ for all $r\geq k$ by (ii), and $u_k=x_k$ implies $x_r=u_r$ for all $r\geq k$. It follows that $\vec x\in N_i(\vec{w})$.

   The proof for $x\in N_j(\vec v)$ is analogous. 
   It therefore follows that   $N_j(\vec u)\subseteq N_i(\vec{w})\cap N_j(\vec{v}) $.
\end{proof}

We denote by $\vec{\mathcal{T}}_\prec$ the topology on $|\vec{\qm}|$ generated by the basis $\mathcal{B}_\prec$.

\begin{lemma}\label{lemmScattered}
   The topological space $(  |\vec{\qm}|,\vec{\mathcal T}_\prec)$ is  scattered.
\end{lemma}
\begin{proof}

 We need to show that for all non-empty $A\subseteq |\vec\qm|$ there exists $\vec x\in A$ and a neighbourhood $O_{\vec x} $ of $\vec x$ such that $O_{\vec x}\cap A=\{\vec x\}$. 
 Suppose $A\subseteq |\vec\qm|$.
 Since by definition $(|\vec{\qm}|,\prec)$ is a locally finite strict partial order, it is well-founded.
 Hence we may choose $\vec x\in A$ such that $x_0$ is $\prec$-minimal among all elements of $A$, and set $O_{\vec x} =N_1(\vec{x})$.
 Observe that $N_1(\vec x)$ contains only $\vec v$ such that $v_0 \preccurlyeq x_0$, which by the minimality of $x_0$ implies that $v_0=x_0$.
 But then, $\vec v $ must belong to $N_1(\vec x)$ by virtue of clause \ref{caseEqual}, which yields $v_i=x_i$ for all $i\geq 0$, i.e.\ $\vec v=\vec x$. Hence, $N_1(\vec{x})\cap A=\{\vec x\}$, as required.
\end{proof}
It is only left to show that the transition system $\sigma$ is continuous. 
\begin{lemma}\label{contmap}
	The shift map $\sigma:|\vec{\qm}|\rightarrow|\vec{\qm}|$ is continuous under the topology $\vec{\mathcal{T}}_\prec$.
\end{lemma}
\begin{proof}
Let $\vec w$ be a realising path and let $N_m(\sigma (\vec w))$ be a neighbourhood of $\sigma(\vec{w})$. By definition, if  $\vec{v}\in N_{m+1}(\vec{w})$ then either $ v_i \prec  w_i $ for all $i<m$, or there exists $k<m$ such that $v_k=w_k$ and this holds henceforth. In any case it follows that $\sigma(\vec v)\in N_m(\sigma (\vec w))$ and in particular $\sigma (N_{m+1}(\vec{w}))\subseteq N_m(\sigma(\vec{w}))$, which implies that $\sigma$ is a continuous function.
\end{proof}
Putting all of the above pieces together, we can now define the limit model of a quasimodel.
\begin{definition}[limit model] Given a quasimodel $\qm=(|\qm|,\prec,\ell,S)$, we define its corresponding \emph{limit model} as a structure
$$ \vec{\qm}=\Big(|\vec{\qm}|,\vec{\mathcal{T}}_\prec,\sigma,\llbracket\cdot\rrbracket^\ell\Big),$$
	where $|\vec{\qm}|$ is the set of realising paths of $\qm$, $\vec{\mathcal{T}}_\prec$ is the topology on $|\vec{\qm}|$ generated by the basis $\mathcal{B}_\prec$, $\sigma$ is the shift operator on $|\vec{\qm}|$ and $\llbracket\cdot\rrbracket^\ell$ is a valuation defined on each propositional variable $p$ as
	$$ \llbracket p \rrbracket^\ell=\{\vec{w}:p\in\ell(w_0)\}.$$
\end{definition}

The key feature of our quasimodel, aside from being a scattered dynamical system, is that it indeed satifies all formulas satisfied by $\qm$.

\begin{lemma}\label{trucon}
Given a quasimodel $\qm$ and a realising path $\vec{w}=(w_n)_{n<\omega}$, if $\varphi\in\ell_\qm(w_0)$, then
$\vec w\in \val{\varphi}^\ell$.
\end{lemma}

\begin{proof}
Let $\qm=(W,\prec ,S ,\ell)$.
	The proof proceeds by a standard induction on the complexity of the formula. The induction steps for $\wedge$ and $\neg$ are routine. We will prove the induction steps of the spatial operator $\lozenge$ and the temporal operators $\blacklozenge$ and $\newmoon$.
	\medskip
		
\noindent {\sc Case $\varphi=\newmoon\psi$ or $\varphi=\neg \newmoon\psi$:} This follows from the fact that $(w_0,w_1)$ is sensible and by the induction hypothesis.
\medskip		
			
\noindent {\sc Case $\varphi=\blacklozenge\psi$:} Since $\vec{w}$ is a realising path, $\blacklozenge\psi\in\ell(w_0)$ implies $\psi\in\ell(w_n)$ for some $n\geq 0$. By the induction hypothesis $\sigma^n(\vec{w})\in \llbracket\psi\rrbracket^\ell$ and so $\vec{w}\in \llbracket \blacklozenge\psi\rrbracket^\ell$.
\medskip		
			
\noindent {\sc Case $\varphi=\blacksquare\psi$:} Since $(w_n,w_{n+1})$ is sensible for all $n$, by a simple induction it follows that $\psi\in\ell(w_n)$. Then, by the induction hypothesis $\sigma^n(\vec{w})\in\llbracket\psi\rrbracket^\ell$, and since $n$ is arbitrary then $\vec{w}\in\llbracket\blacksquare\psi\rrbracket^\ell$. 
\medskip		
			
\noindent {\sc Case $\varphi=\lozenge\psi$}:
Suppose $\lozenge\psi\in\ell(w_0)$ and let $N=N_m(\vec w)$ be a neighbourhood of $\vec w$.
Since $\ell$ is  a labelling function, $\psi\in \ell(v_0)$ for some $v_0$ such that $v_0\prec w_0$.
By Lemma \ref{paths}, there is a path $( v_i)_{i<n}$ such that $v_i\peq w_i$.
If $v_k=w_k$ for some (least) $k<m$, then define $v_i=w_i$ for all $i>k$ (redefining values if needed).
Otherwise, apply Lemma \ref{extensive} to extend $( v_i)_{i<n}$ to a realising path $\vec v$.
In either case, it is readily checked that $\vec v\in N$.
By the induction hypothesis we get $\vec{v}\in \llbracket \psi\rrbracket^\ell $. 
Then, from Lemma \ref{paths} and the definition of the topology $\vec{\mathcal T}_\prec$ we get $\vec{w}\in \llbracket \lozenge\psi\rrbracket^\ell$, as required.
\medskip		
			
\noindent {\sc Case $\varphi=\Box\psi$:}
Suppose that $\Box\psi\in \ell(w_0)$.
By the semantics of $\Box$, we need to find a neighbourhood $N$ of $\vec w$ such that $N\setminus \{\vec w\}\subseteq \val\psi^\ell$.
We propose $N=N_1(\vec w)$.
Then, if $\vec v\in N\setminus \{\vec w\}$, we must have that $v_0\prec w_0$ (as $v_0=w_0$ forces $\vec v=\vec w$ by the definition of $N$).
Since $\ell$ is a labelling function, $ \psi\in\ell(v_0)$, hence the induction hypothesis yields $\vec v\in \val\psi^\ell $, as needed.
	\end{proof}
We are now ready to prove the main result of this section.
\begin{theorem}\label{quastolim}
Let $\qm$ be a quasimodel. 
Then $\vec{\qm}$ is a scattered dynamical model, and if $\qm$ satisfies $\varphi$, it follows that $\vec{\qm}$ satisfies $\varphi$.
\end{theorem}
\begin{proof}
By Lemma \ref{lemmScattered}, $(|\vec{\qm}|,\vec{\mathcal{T}}_\prec)$ is a scattered space.
	It is clear that $\sigma$ is a function, and by Lemma \ref{contmap} it is continuous.
	It is therefore the case that $(|\vec{\qm}|,\vec{\mathcal{T}}_\prec,\sigma)$ is a scattered dynamical system. Suppose $w_0\in |\qm|$ and $\varphi\in \ell(w_0)$. 
	Then, by Lemma \ref{extensive}, $w_0$ can be extended to a realising path $\vec{w}$. 
	From Lemma \ref{trucon} we get that $\vec{w}\in\llbracket\varphi\rrbracket^\ell$, hence $ \varphi$ is satisfied in $\vec{\qm}$. 
\end{proof}

\begin{corollary}\label{corQSound}
Every formula satisfiable on a quasimodel is satisfiable on a scattered dynamical model.
\end{corollary}

Our strategy for the remainder of the completeness proof will therefore be to show that if $\varphi$ is consistent, then it is satisfiable on a quasimodel: from Corollary \ref{corQSound}, this suffices to ensure that it is indeed satisfiable on the class of scattered dynamical systems.

\section{Simulating states and simulation formulas}\label{simulstatsimulform}
In this section we introduce the notion of $\Sigma$-\emph{states,} which are local descriptions of quasimodels very similar to finite $\Sigma$-labelled posets but with a root. 
The $\Sigma$-\emph{states} form the universe $|\univ|$ of the \emph{universal state space} $\univ=(|\univ|,\prec,\mapsto,\ell)$ of $\Sigma$, which will be used in order to establish the connection between the semantic framework of the limit models with the syntactic derivations in $\mathbf{DGL}$.

The structure $\univ$ is universal in the sense that every model can be {\em simulated} by a $\Sigma$-state $\ww\in |\univ|$.
Simulations are the correct notion of `embedding' from the point of view of modal logic, just as bisimulations are the correct notion of `isomorphism.'
In the context of labelled structures, this notion is defined as follows.

\begin{definition}[labelled simulation] Given two labelled posets $\sfont A $ and $\sfont B $, a relation $R\subseteq |\sfont A|\times|\sfont B|$ is {\em strictly forward-confluent} if $a'\prec_\sfont A a$ and $aRb$ implies that there is $b'\prec_\sfont B b$ such that $a'Rb' $.
A \emph{labelled simulation} is a strictly forward-confluent relation $\chi\subseteq |\sfont A|\times |\sfont B|$ such that $w\chi v$ implies $\ell_\sfont A(w)=\ell_\sfont B(v)$.
\end{definition}

After defining $\Sigma$-states, we will show that for each $\Sigma$-state $\ww$ there is a formula $\Sim(\ww)$ defining the property of being simulated by $\ww$. We will then prove that certain derivations in regards to $\Sim(\ww)$ are possible whenever some relevant conditions on $\ww$ hold.
This part relies on the completeness and finite model property of $\mathbf{GLC}$ (Theorem \ref{glcomp}).
We later use this information to define the consistent restriction of $\univ$ and to show that this restriction is a `canonical' quasimodel.

\subsection{Simulating states}\label{secSimStat}

We first define the `worlds' of our universal structure, which we call `states'.

\begin{definition}[state]
A {\em state} is a tuple
$$ \ww=( |\ww|,\prec_\ww,\ell_\ww,0_\ww),$$
where $( |\ww|,  \prec_\ww,\ell_\ww )$ is a finite labelled poset and $0_\ww$ is a distinguished point such that $v\prec0_\ww$ for all $v\in|\ww|$.

If $\Sigma$ is a set of formulas such that the range of $\ell_\ww $ is contained in $\mathbb T_\Sigma$, we say that $\ww$ is a {\em $\Sigma$-state.}
\end{definition}




The set of all states is infinite, but it is essential that each individual state be finite.
The following definition provides a useful way to measure the size of each state.

\begin{definition}[norm]
	Given a $\Sigma$-state $\ww$ we denote by $\mathbf{hgt}(\ww)$ the maximum length of a $\prec $-sequence of points in $|\ww|$. Moreover, we denote by $\mathbf{wdt}(\ww)$ the maximum $N$ such that there exists $w\in|\ww|$ with $N$ daughters which are pairwise $\prec$-incomparable.
	
	 The \emph{norm} of $\ww$ is then defined as
	 $$ \|\ww\|=\mathbf{max}(\mathbf{hgt}(\ww),\mathbf{wdt}(\ww)).$$
\end{definition}

Being labelled structures, the notion of simulation readily applies to states, with the caveat that all simulations must be root-preserving in this context.

\begin{definition}[simulations between states]
Let $\ww$ and $\mathfrak v$ be $\Sigma$-states. We say that $\ww$ {\em simulates} $\mathfrak v$ if there exists a labelled simulation $\chi\subseteq |\ww|\times |\mathfrak v|$ such that $0_\ww \chi 0_\mathfrak v$. We write $\ww \vartriangleleft \mathfrak v$ if $\ww$ simulates $\mathfrak v$.
\end{definition}

Note that compositions of simulations are simulations, given that compositions of strictly forward-confluent relations are also strictly forward-confluent.
Thus the relation $\vartriangleleft$ is transitive.
Since the identity is a simulation, it is also reflexive.
Thus $\vartriangleleft$ is a quasiorder on the set of states.
This relation will be essential in controlling the size of states we must consider, as when $\ww\vartriangleleft \vv$, it is often the case that $\vv$ can be replaced by $\ww$ as far as satisfiability is concerned, even when the latter is much smaller. 

\subsection{The universal state space}

Given a set of formulas $\Sigma$, the set of $\Sigma$-states forms a weak $\Sigma$-quasimodel.
In order to see this, we first need to equip the set of $\Sigma$-states with a suitable strict partial order.

Below, we say that a $\Sigma$-state $\vv$ is a \emph{generated substructure} of a $\Sigma$-state $\ww$ if $|\vv|$ is a downward-closed subset of $|\ww|$ with respect to $\prec_\ww$, such that ${\prec_\vv}={\prec_\ww}\cap(|\vv|\times|\vv|)$ and $\ell_\vv(v)=\ell_\ww(v)$ for all $v\in|\vv|.$

\begin{definition}[substate]
	Let $\ww$ and $\vv$ be $\Sigma$-states. 
	We call $\mathfrak v$  a \emph{substate} of $\ww$ and denote it by $\mathfrak v\prec \ww$ if $0_\mathfrak w \neq 0_\mathfrak v\in |\ww|$ and $\mathfrak v$ is a generated substructure of $\ww$.
\end{definition}

We write $\ww\mapsto\vv$ if there exists a sensible relation $R\subseteq |\ww|\times |\mathfrak v|$ such that $0_\ww R 0_\vv$.
We say that $\vv$ is a \emph{bounded future} of $\ww$ and denote it by $\ww\mleadsto\vv$, if $\ww\mapsto\vv$ and in addition the following inequality is satisfied:
$$ \|\vv\|\leq \|\ww\|+\Big|\bigcup_{w\in|\ww|}\{\lozenge\varphi \in \ell_\ww(w)\} \Big|.$$

\begin{definition}[universal state space]
Let $\Sigma\Subset \fulllan$ and fix $K\geq 0$. We define $|\univ^K |$ to be the set of all $\Sigma$-states $\ww$  for which $\|\ww\|\leq (K+1)\cdot| \Sigma |$.

We denote by $|\univ|$ the union $\bigcup_{k<\omega}|\univ^k | $, and we use it to define the \emph{universal state space} 
$$ \univ=(|\univ|,\prec,\mapsto,\ell),$$
where $\ell(\ww)=\ell_\ww(0_\ww)$.
\end{definition}
The universal state space has several desirable properties that we are interested in. 
\begin{proposition}\label{unistaprop}
	Let $\Sigma\subset\fulllan$ be a finite set of formulas. 
	Then, for every $\Sigma$-state $\ww$ the following conditions are satisfied:
	\begin{enumerate}
		\item There exists $
		\vv\in|\univ^0|$ such that $\vv\vartriangleleft\ww$;
		\item\label{itUnisTwo} if $\ww\mapsto \vv$ for some $\vv\in |\univ|$, then there is $\mathfrak u\vartriangleleft \vv$ such that $\ww \mleadsto\mathfrak u$.
	\end{enumerate}
\end{proposition}
\begin{proof}
	The proof proceeds by induction on the height of $\ww$. 
	It follows a similar proof from \cite{FernandezNonDeterministic}.
\end{proof}

We call a nonempty set $A\subseteq |\univ|$ \emph{regular} if it is open and the restriction $\mapsto\upharpoonright_A$ is $\omega$-sensible. 
By the definition of a quasimodel we obtain the following:
\begin{proposition}\label{regular}
	If $A\subseteq |\univ|$ is regular, then $\univ{\upharpoonright_A}$ is a quasimodel. 
 \end{proposition}
We call quasimodels of the form $\univ{\upharpoonright_A}$, where $A$ is regular, \emph{regular quasimodels}.
\subsection{Simulation formulas}\label{secSimFor}
Next, we introduce the formulas $\Sim(\ww)$, which define the property of being simulated by $\ww$.
Recall from Example \ref{examModisQ} that if $\sfont M $ is a model, then for $x\in |\sfont M|$ we defined $\ell_\Sigma(x) = \{\varphi\in \Sigma: x\in \val\varphi_\sfont M\}$, and that $\sfont M$ is thus identified with the corresponding quasimodel.
Thus the proposition below applies to both models and to (weak) quasimodels.

 \begin{proposition}[simulation formulas]\label{topstate}
	Let $\ww$ be a $\Sigma$-state.
	Then there exists a formula $\mathbf{Sim}(\ww)$ such that for every scattered dynamic model $\sfont M$ and $x\in |\sfont M|$, we have that
	\[x\in \val{\mathbf{Sim}(\ww)}_\sfont M  \Leftrightarrow \ww \vartriangleleft (\sfont M,x).\]
	Moreover, $\mathbf{Sim}(\ww)$ can be  defined as  
	$$ \mathbf{Sim}(\ww):=\bigwedge \ell_\ww(0_{\ww}) \wedge \bigwedge_{\vv\prec  \ww}  \lozenge \mathbf{Sim}(\vv).$$
	\end{proposition}
		
\begin{proof}
Note that ${\downarrow \ww}  = \{\vv: \vv\prec \ww\} $ is in bijection with $|\ww|$ via the map $(\hat\cdot) \colon |\ww| \to {\downarrow} \ww $, where $\hat v$ is the unique state such that $\hat v\prec \ww$ and $0_{\hat v}= v$.
It follows that the set $ {\downarrow} \ww $ is finite, hence the relation $\prec$ is well-founded and $\Sim(\ww)$ is well-defined.
With this, we prove the above items.
\medskip

\noindent	$(1)$ Suppose $ \mathfrak M,x \models \mathbf{Sim}(\ww) $. 	
	 We define a relation $\chi \subseteq |\ww| \times \sfont A $ by setting $v\chi z $ iff $ \mathfrak M,z \models \Sim(\hat v)  $.
	It is not hard to check using the definition of $\Sim(\cdot)$ that $\chi$ is strictly forward-confluent, and that for each pair $v \chi z$, we have $\ell_\ww ( v)= \ell(\hat v) = \ell_{\Sigma}(x)$ as required.
	\medskip
	
\noindent $(2)$ Suppose that $ \sfont M,x \models \neg\Sim(\ww) $.
We suppose that $\chi$ is a simulation such that $0_\ww \chi x $ and show that it leads to a contradiction. 
By definition, either $ \sfont M,x \models \neg \bigwedge \ell_\ww(0_{\ww})   $ or $ \sfont M,x \models \neg \bigwedge_{\vv\prec  \ww}  \lozenge \mathbf{Sim}(\vv) $.
In the first case, we cannot have $\ell(\ww) = \ell_\Sigma(x)$, contradicting that simulations preserve labels.
In the second, note that there exists $v\in|\ww|$ such that ${v \prec_\ww 0_\ww}$ and $$ \sfont M,x \models  \neg \lozenge  \Sim(v) . $$
It follows that if $y \prec_\sfont M x$, then $\sfont M,y \models  \neg \Sim(v)$, so that the induction hypothesis yields $\neg(v \chi y) $.
Since $y\prec_\sfont M x$ was arbitrary, we conclude that $\chi$ cannot be strictly forward-confluent, contradicting the assumption that it is a simulation.
\end{proof}

There are a few important derivable properties that hold in relation to simulation formulas and that should be discussed before we proceed to the main part of the proof. 
Below, recall that $\Sigma\Subset \fulllan$ means that $\Sigma$ is finite and closed under subformulas and single negations.

\begin{lemma}\label{simone}
	Let $\Sigma\Subset \fulllan$ and $\ww=( |\ww|,\prec,\ell,0_\ww)$ be a $\Sigma$-state. Then the formula $\mathbf{Sim}(\ww)$ satisfies the following properties:
	\begin{enumerate}
		\item If $\varphi\in \ell(\ww)$, then $\vdash\mathbf{Sim}(\ww)\rightarrow\varphi;$
		\item if $\vv\vartriangleleft\ww$ then $\vdash\Sim(\ww)\rightarrow\Sim(\vv)$;
		\item\label{third} if $\vv\prec\ww$ then $\vdash\Sim(\ww)\rightarrow\lozenge\Sim(\vv)$;
		\item if $\varphi\in  \Sigma $, then 
		$$ \vdash\varphi\rightarrow\bigvee_{\substack{\ww\in\univ^0, \\ {\varphi\in\ell(\ww)} }} \Sim(\ww);$$
		\item for all $\ww\in \univ$,
		$$ \vdash\Sim(\ww)\rightarrow\newmoon\bigvee_{\ww\mleadsto\vv}\Sim(\vv).$$
	\end{enumerate}
\end{lemma}
\begin{proof} We show in order that each of the formulas above is a $\mathbf{GLC}$ validity. By the completeness of $\mathbf{GLC}$ (Theorem \ref{glcomp}) this implies that they are derivable in $\mathbf{DGL}$.
	~\paragraph{1.} Suppose $\varphi\in \ell (\ww)$.
	 By the definition of a simulation, if $\sfont M$ is a dynamic poset model such that $\sfont M \models\mathbf{Sim}(\ww)$, then $\sfont M \models \ell(\ww)$.
	  Therefore $\Sim(\ww)\rightarrow\varphi$ is a validity of $\mathbf{GLC}$, which by Theorem \ref{glcomp} yields that $\vdash \Sim(\ww)\rightarrow\varphi$.

	\paragraph{2.} Let $\sfont M$ be a dynamic poset model such that $ \sfont M,x \models\Sim(\ww)$ and $\vv\vartriangleleft \ww$. 
	By  Proposition \ref{topstate}, it follows that $\ww\vartriangleleft(\sfont M,x)$ and so by the transitivity of $\vartriangleleft$ we derive that $\vv\vartriangleleft(\sfont M,x)$. 
	Using Proposition \ref{topstate} once more yields $ \sfont M,x \models \Sim(\vv)$, as needed.
	
	\paragraph{3.} Let $\sfont M$ be a dynamic poset model such that $ \sfont M,x \models\Sim(\ww)$ and $\vv\prec\ww$. Then there is a simulation $\chi\subseteq |\ww|\times|\sfont M|$ such that $0_\ww\chi x$.
    Since $\vv$ is a substate of $\ww$, by the definition of a simulation there is $y\prec_\sfont M x$ such that $0_\vv \chi y$.
    Therefore $\vv\vartriangleleft (\sfont M,y)$ and by Proposition \ref{topstate}, we get $ \sfont M,y \models \Sim(\vv)$. This implies that $ \sfont M,x \models \lozenge \Sim(\vv)$. 
	
	\paragraph{4.} Suppose that $\varphi\in\Sigma$. 
	In order to use the completeness of $\mathbf{GLC}$ in this part of the proof, we need to find a way to convert a $\mathbf{DGL}$ formula to a $\mathbf{GLC}$ formula. 
	We do this by replacing the  `henceforth' outermost appearances of the form $\blacksquare\varphi\in\Sigma$ with a new propositional variable $p_\varphi$. 
	For each $A\subseteq\Sigma$ we denote the resulting set under such operation by $A^{p}$ and the reverse operation by $A^\blacksquare$. 
	
	We prove that since $\mathbf{GLC}$ has the finite model property (Theorem \ref{glcomp}), it is sufficient to show that the formula 
	$$\zeta:=\varphi^{p}\rightarrow\bigvee\{\Sim(\ww): \ww\in\univnp^0_{\Sigma^{p}} \text{ and } \varphi^p\in\ell(\ww)\}	$$
is valid on every finite $\mathbf{GLC}$-model $\sfont M=(X,\tau,f,\nu)$. 
	
	\sloppy Suppose that $\sfont M,w \models\varphi^p$ for some $w\in X$. 
\sloppy	We define a new $\Sigma^{p}$-state $\vv=(|\vv|,\prec,\ell,0_\vv)$, where 
	\begin{itemize}
		\item $|\vv|=\{u\in X:  u\peq_\sfont M w \}$;
	 \item  $0_\vv=w$;
		\item  $\ell (u)=\{\psi\in ( \Sigma )^p: u\in \val\psi _\sfont M\}$, for all $u\in W$.
	\end{itemize}
	\sloppy
	We call this the {\em $\Sigma^{p}$-state associated to $w$.}
	By Proposition \ref{unistaprop} there is $\ww\in\univnp^0_{\Sigma^p}$ such that $\ww\vartriangleleft\vv$ and so by Proposition \ref{topstate} it follows that $w \in \val{\Sim(\ww)}_\sfont M$, hence $\sfont M,w$ satisfies $\zeta$.
	 Since $w\in X$ is arbitrary, $\zeta$ is valid in $\sfont M$.
	
	Next, consider the formula  
	$$ \zeta^\blacksquare=\varphi\rightarrow\bigvee\{\Sim(\ww)^\blacksquare: \ww\in\univnp^0_{\Sigma^{p}} \text{ and } \varphi^{p}\in\ell(\ww)\},$$
	given by substituting back all of the new propositional variables of the form $p_\varphi$.
	This is derivable in $\mathbf{DGL}$ by substitution, which is readily checked to be admissible.
	We may also apply the $(\cdot)^\blacksquare$ operation to $\Sigma^{p}$-states by applying it to each label.
Note that for some $\Sigma^p$-state $\ww$ with $\varphi^p\in\ell (\ww)$ it might be that $\ww^\blacksquare$ is not a $\Sigma$-state;
	 i.e.\ for some $w\in |\ww|$, the labelling $\ell(w)^{\blacksquare}$ is not a $\Sigma$-type.
	 However, we prove that such simulation formulas are inconsistent and thus can be removed from the disjunction.
	
	By the definition of $\Sigma$-type (Definition \ref{sigtype}) the only case that we need to consider is where $\blacksquare\varphi\in\ell(w)^\blacksquare$ but $\varphi\notin\ell(w)^\blacksquare$. 
	Since $\ell(w)$ is a $\Sigma^p$-type and $\varphi^p\in \Sigma^p$, while $\varphi^p\notin \ell(w)$, $\neg\varphi^p\in\ell(w)$ and therefore $\neg\varphi\in \ell(w)^\blacksquare$. But since $\vdash\blacksquare\varphi\rightarrow\varphi$, also $\vdash\bigwedge\ell(w)\rightarrow\bot $. Therefore $\ell(w)$ is inconsistent. 
	
	Suppose $\vv\peq \ww$ is such that $0_\vv= w$. Then by the already proven item \ref{third}, it follows that $\vdash\Sim(\ww)\rightarrow\Sim(\vv)\vee\lozenge \Sim(\vv)$. In addition, by the already proven item 1, we have $\vdash \Sim(\vv)\rightarrow\bigwedge\ell(w)$. But since $\ell(w)$ is inconsistent, so is $\Sim(\ww)$, i.e\ $\vdash\neg\Sim(\ww)$. By removing from the disjunction of $\zeta^\blacksquare$ all such inconsistent simulation formulas we obtain the required result.
	
	\paragraph{5.} Let $\Phi\in\mathbb T_\Sigma$ be a $\Sigma$-type. 
	We define $$\Phi^+=\Phi\cup\{\newmoon\blacksquare\psi:\blacksquare\psi\in \Sigma\}\cup \{\newmoon\blacklozenge\psi: \blacklozenge\psi\in\Phi \text{ and } \psi\not\in\Phi\}.$$
For a state $\mathfrak x$, let $\mathfrak x^+$ be a state identical to $\mathfrak x$ but with $\ell_\mathfrak x(w)$ replaced by $\ell^+_\mathfrak x(w):=(\ell_\mathfrak x(w))^+$ for all $w\in|\mathfrak x|$.

	From the axioms of $\mathbf{DGL}$, it is clear that 
	\begin{equation}\label{plusequiv}
		\vdash\bigwedge\Phi\leftrightarrow\bigwedge\Phi^+.
	\end{equation}
It is routine to check that $\vdash \Sim(\mathfrak x)\leftrightarrow \Sim(\mathfrak x^+)$ as well for every $\Sigma$-state $\mathfrak x$.
Thus we prove that
$$ \vdash\Sim(\ww^+)\rightarrow\newmoon\bigvee_{\ww\mleadsto\vv}\Sim(\vv).$$ 

As before, $\psi^p$ is a formula where every outermost formula $\blacksquare\theta$ is replaced by $p_\theta$.
We write $\Psi^{+p}$ instead of $(\Psi^+)^p$, and define $\ww^{+p}$ similarly.
Let $\sfont M$ be any finite dynamic poset model and suppose that $x\in \val{ \Sim(\ww^{+p}) } $.
Reasoning as above, we have that $\ww^{+p} \lhd (\sfont M,x)$.
Let $\vv_0$ be the $\Sigma$-state associated with $x$ and let $\vv_1$ be the $\Sigma$-state associated with $S_\sfont M(x)$, as defined in the previous item.
Then, the function $S_\sfont M$ witnesses that $\vv_0 \mapsto \vv_1 $, hence since $\ww ^{+p} \lhd \vv_0$ and by using the fact that the composition of forward-confluent relations is forward-confluent, we have that $\ww ^{+p}\mapsto \vv_1$.
By Lemma \ref{unistaprop}.\ref{itUnisTwo}, there is some $\vv$ such that $\ww ^{+p}\mleadsto \vv  \vartriangleleft \vv_1$.
Since $\sfont M$ was arbitrary, we obtain
\begin{equation}\label{eqSimPlus}
\mathbf{GLC} \vdash\Sim(\ww^{+p})\rightarrow\newmoon\bigvee_{\ww^{+p}\mleadsto\vv }\Sim(\vv ).
\end{equation}
Using the fact that $\vdash\Sim(\ww)\leftrightarrow  \Sim(\ww^{+ }) $ and $\Sim(\ww^{+ }) = (\Sim(\ww^{+p}))^\blacksquare$, we may apply $(\cdot)^\blacksquare$ to \eqref{eqSimPlus} in order to obtain
$$\vdash\Sim(\ww )\rightarrow\newmoon\bigvee_{\ww^{+p}\mleadsto\vv } \Sim(\vv  ^\blacksquare).$$
As above, those instances of $ \Sim(\vv ^\blacksquare ) $ where $\vv^\blacksquare$ is not a $\Sigma$-state are inconsistent.
If instead $\ww^{+p} \mleadsto \vv $ and $\vv^\blacksquare$ is a $\Sigma$-state, it is not hard to check from the way we defined $(\cdot)^+$ that $\ww  \mleadsto \vv^\blacksquare $, as the extra formulas added to $\ww^{+}$ ensure that the sensibility conditions for $\blacklozenge$ and $\blacksquare$ are satisfied.
Thus we obtain that
$$  \vdash\Sim(\ww )\rightarrow\newmoon\bigvee_{\ww \mleadsto\vv } \Sim(\vv ),$$
as needed.
\end{proof}

\section{Canonical quasimodels}\label{canquas}

In this section we focus on constructing a canonical quasimodel for $\Sigma$. 
We denote it by $\cqm_\Sigma$, which we temporarily dub the \emph{canonical structure} of $\Sigma$. 
It is the restriction of $\univ$ to consistent states, i.e.\ states $\ww$ for which $\vdash\Sim(\ww)$. 
We prove that $\cqm_\Sigma$ is a quasimodel by showing that $\mapsto$ is serial and $\omega$-sensible. 

Once we have all the required results, we can conclude that $\mathbf{DGL} $ is complete by showing that every consistent formula $\varphi$ yields a consistent state $\ww\in\cqm_{\Sigma}$, where $\Sigma=\sub_\pm(\varphi)$.
Since $\cqm_\Sigma$ is a quasimodel, $\vec{\qm}_\Sigma \models\varphi$. Since $\vec{\qm}_\Sigma$ is a scattered dynamical model, the logic $\mathbf{DGL}$ is complete with respect to such models.

\subsection{The canonical structure}
 
	We say that a $\Sigma$-state $\ww$ is \emph{inconsistent} if $\vdash\neg\mathbf{Sim}(\ww)$; otherwise it is consistent. 
	The set of consistent $\Sigma$-states is denoted by $\mathsf{Cons}(\Sigma)$.
\begin{definition}[canonical structure]
	For a set of formulas $\Sigma$, we define the canonical structures of $\Sigma$ as the quadruple $\cqm_\Sigma=(|\cqm_\Sigma|,\prec_{\cqm_\Sigma},\mapsto_{\cqm_\Sigma},\ell_{\cqm_\Sigma})$, where
\begin{itemize}
	\item $|\cqm_\Sigma|=\mathsf{Cons}(\Sigma)$;
	\item ${\prec}_{\cqm_\Sigma}={\prec_{\univ}}\cap(\mathsf{Cons}(\Sigma)\times \mathsf{Cons}(\Sigma)) $;
	\item ${\mapsto}_{\cqm_\Sigma}= {\mleadsto_{\univ}}\cap (\mathsf{Cons}(\Sigma)\times \mathsf{Cons}(\Sigma))$;
	\item $\ell_{\cqm_\Sigma}=\ell_{\univ}\cap (\mathsf{Cons}(\Sigma)\times \wp(\fulllan)).$
\end{itemize}
\end{definition}
\begin{lemma}\label{openserial}
Let $\Sigma\Subset\fulllan $. Then $|\cqm_\Sigma|$ is open in $|\univ|$ and $\mapsto_{\cqm_\Sigma}$ is serial.
\end{lemma}
\begin{proof}
We show that the the properties are preserved in the new structure $\cqm_\Sigma$.
Suppose that $\ww\in |\cqm_{\Sigma}|$, i.e.~$\ww$ is a consistent $\Sigma$-state.
	
	Let $\vv\prec_{\univ}\ww$, meaning $\vv$ is a substate of $\ww$.
	 By Proposition \ref{simone}.3 we derive that $\vdash \mathbf{Sim}(\ww)\rightarrow\lozenge \mathbf{Sim}(\vv)$ and so if $\ww$ is consistent, then so is $\vv$.
	It follows that $\vv \in |\cqm_\Sigma|$, and so $|\cqm_\Sigma|$ is open. 
	
	By Proposition \ref{simone}.5 we have $\vdash\mathbf{Sim}(\ww)\rightarrow\newmoon\bigvee_{\ww{\mleadsto}\vv}\mathbf{Sim}(\vv)$. The consistency of $\ww$ implies that there exists some $\vv$ for which $\ww\mleadsto\vv$ and $\vv$ is consistent as well. It follows that $\vv\in|\cqm_\Sigma|$, hence $\mapsto_{\cqm_\Sigma}$ is serial.
\end{proof}
	\subsection{Efficiency and $\omega$-sensibility}\label{sectCanQuasSig}
	There is a point of tension that we need to address before proceeding.
	We need to be able to determine when a formula of the form $\blacklozenge\varphi$ will be realised, which becomes difficult as there is an infinite number of $\Sigma$-states to consider.
	We deal with this by showing that it is sufficient to consider a finite set of \emph{efficient} paths, which allows us to only consider finitely many states when evaluating each instance of $\blacklozenge\varphi$.
	
	In the following, we let $\vec{\ww}=(\ww_n)_{n\leq \alpha}$ denote a finite path of $\Sigma$-states. 
\begin{definition}[efficiency]
A finite path $\vec{\ww}$ is called \emph{efficient} if the following conditions are satisfied:
\begin{enumerate}
	\item For all $n<\alpha$, $\ww_n\mleadsto\ww_{n+1}$;
	\item for all $i<j$ and states $\ww_i,\ww_j$ in the path $\vec{\ww}$, $\ww_i\not\vartriangleleft \ww_j$.
\end{enumerate}
	
\end{definition}
In order to show that there is a finite number of efficient paths that are to be considered for each $\Sigma$-state, we  will introduce and utilise \emph{Kruskal's theorem}.

A $\Sigma$-\emph{labelled tree} is a triple $\mathfrak T=(T,\leq,\ell)$, where $(T,\leq)$ is a tree and $\ell$ is a $\Sigma$-labelling function. 
An injective map $\iota:T_1\rightarrow T_2$ between two finite $\Sigma$-labelled trees $\mathfrak T_1$ and $\mathfrak T_2$ is called an \emph{embedding} if for all $x,y\in T_1$, $x\leq_1 y$ if and only if $\iota(x)\leq_2 \iota(y)$, and in addition $\ell_1(x)=\ell_2(\iota(x))$.
\begin{theorem}[Kruskal's tree theorem]\label{krusor}
For every infinite sequence $\mathfrak T_0, \mathfrak T_1,\dots$ of finite labelled $\Sigma$-trees there are indices $i<j<\omega$ for which there exists an embedding $\iota:\mathfrak T_i\rightarrow\mathfrak T_j$. 
\end{theorem}
\begin{proof}
This can be found in Kruskal's original paper \cite{kruskal}.
\end{proof}
We would like to use Kruskal's theorem on states. We can do so by observing that each state is bisimilar to a finite tree, and bisimulation preserves simulability.
Thus we obtain the following result:

\begin{lemma}\label{krus}
For every infinite sequence $\ww_0,\ww_1,\dots$ of $\Sigma$-states there are indices $i<j<\omega$ such that $\ww_i\vartriangleleft\ww_j$.
\end{lemma}
We can now prove that there is a bound on the number of efficient paths with the same root. 
\begin{proposition}\label{fineff}
	Let $\vv$ be a $\Sigma$-state. There are finitely many efficient paths $\vec{\ww}$ such that $\ww_0=\vv$.
\end{proposition}
\begin{proof}

Suppose the contrary. Then by K\"onig's Lemma we have an infinite path on the tree of efficient paths, i.e.\ an increasing sequence of efficient paths starting at $\ww_0$. 
This increasing sequence yields an infinite path $(\ww_0,\ww_1,...)$, all of whose initial segments are efficient. 
Since this path is infinite, by Lemma \ref{krus}, there are indices $i<j<\omega$ such that $\ww_i \lhd \ww_j$.
Hence the finite initial segment $(\ww_0,...,\ww_j)$ is inefficient in contradiction.
\end{proof}

With this we define a notion of reachability which refines the transitive, reflexive closure of $\mleadsto$.

\begin{definition}[efficient reachability]
	Let $\ww$ be a $\Sigma$-state. A $\Sigma$-state $\vv$ is \emph{efficiently reachable} from $\ww$ if there exists a finite efficient path $\vec{\mathfrak p}=(\mathfrak p_0,\dots,\mathfrak p_\alpha)$ of consistent states such that $\mathfrak p_0=\ww$ and $\mathfrak p_\alpha=\vv$.
	
	We denote by $\varrho(\ww)$ the set of states that are efficiently reachable from $\ww$.
\end{definition}
\begin{lemma}\label{firstdis}
	For every $\ww\in |\cqm_\Sigma|$, the set $\varrho(\ww)$ is finite.  
\end{lemma}
\begin{proof}
	  This follows directly from Proposition \ref{fineff}. 
\end{proof}

We will use this result to ensure that the formulas in Lemma \ref{conspres} and Lemma \ref{omegasens} below have finite disjunctions and hence are well defined.

The following derivation is required for showing that $\mapsto_{\cqm_\Sigma}$ is $\omega$-sensible:
\begin{lemma}\label{conspres}
	Let $\ww\in|\cqm_\Sigma|$. Then \begin{equation}\label{conspresfor}
 	\vdash\bigvee_{\vv\in\varrho(\ww)}\mathbf{Sim}(\vv)\rightarrow \newmoon \bigvee_{\vv\in\varrho(\ww)}\mathbf{Sim}(\vv).
 \end{equation}

\end{lemma}
\begin{proof}
\sloppy
From Proposition \ref{simone}.5 it follows that $\vdash\mathbf{Sim}(\vv)\rightarrow\newmoon\bigvee_{\vv\mathfrak\mleadsto \mathfrak u}\mathbf{Sim}(\mathfrak u)$, for all $\vv\in \varrho(\ww)$.
We may remove all inconsistent states from the disjunction to obtain $\vdash\mathbf{Sim}(\vv)\rightarrow\newmoon\bigvee_{\vv \mapsto_{\cqm_\Sigma} \mathfrak u}\mathbf{Sim}(\mathfrak u)$.
We claim that for each such $\mathfrak u$, there is $\mathfrak u'\in \varrho(\ww)$ such that $\vdash \Sim(\mathfrak u)\to \Sim(\mathfrak u')$, so that we may obtain $\vdash\mathbf{Sim}(\vv)\rightarrow\newmoon\bigvee_{\vv\in\varrho(\ww)}\mathbf{Sim}(\mathfrak u)$, as required.

Let $\vec{\mathfrak p} = (\mathfrak p_0,\ldots,\mathfrak p_{\alpha-1})$ be an efficient path from $\ww$ to $\vv$. 
We know that such a path exists since $\vv$ is efficiently reachable from $\ww$.
Let $\mathfrak u$ be such that $ \vv \mapsto_{\cqm_\Sigma} \mathfrak u $ and let $\vec{\mathfrak p}'$ be the same as $\mathfrak p$ only that we add a last element $\mathfrak p'_\alpha=\mathfrak u$.

If $\mathfrak u\in\varrho(\ww)$, there is nothing to prove.
Otherwise, $\vec{\mathfrak p}'$ cannot be an efficient path, since it would witness that $\mathfrak u$ is indeed reachable from $\ww$.
Since $\vec{\mathfrak p}$ is an efficient path, this can only occur if there is some $n<\alpha $ such that ${\mathfrak p}'_n \vartriangleleft{\mathfrak p}'_\alpha$.
By Proposition \ref{simone}.2 we have $\vdash\mathbf{Sim}(\mathfrak u) \to \mathbf{Sim}(\mathfrak p'_n)$.
Thus $\mathfrak p'_n$ is the desired value of $\mathfrak u'$.
  
  Since we took $\vv$ to be some ${\vv\in \varrho(\ww)}$ without further specifications, we can combine this to get ${\vdash\bigvee_{\vv\in\varrho(\ww)}\mathbf{Sim}(\vv)\rightarrow \newmoon \bigvee_{\vv\in\varrho(\ww)}\mathbf{Sim}(\vv)}$, as required.
\end{proof}
We are now ready to prove that $\mapsto_{\cqm_\Sigma}$ is $\omega$-sensible.
\begin{lemma}[$\omega$-sensibility]\label{omegasens}
Let $\ww\in|\cqm_\Sigma|$ and $\blacklozenge\varphi\in\ell(\ww)$. Then there is $v\in\varrho(\ww)$ such that $\varphi\in \ell(\vv)$.	
\end{lemma}
\begin{proof}
	We prove this by contradiction. Suppose $\ww\in|\cqm_\Sigma|$ and $\blacklozenge\varphi\in\ell(\ww)$, while there exists no $v\in\varrho(\ww)$ with $\varphi\in\ell(\vv)$. Using the formula (\ref{conspresfor}) from Lemma \ref{conspres} together with the axioms $\mathrm{Nec}_\blacksquare$ and $\mathrm{Ind}_\blacksquare$, we get that
	\[\vdash\bigvee_{\vv\in\varrho(\ww)}\mathbf{Sim}(\vv)\rightarrow \blacksquare \bigvee_{\vv\in\varrho(\ww)}\mathbf{Sim}(\vv) .\]
	Since clearly $\ww$ is efficiently reachable from itself, 
	\begin{equation}\label{hyposyl}
	\vdash\mathbf{Sim}(\ww)\rightarrow \blacksquare \bigvee_{\vv\in\varrho(\ww)}\mathbf{Sim}(\vv). 	
	\end{equation}
	Suppose $\vv\in\varrho(\ww)$. Since $\vv$ is a $\Sigma$-state, by the assumption that $\varphi\notin\ell(\vv)$ we obtain $\neg\varphi\in\ell(\vv) $.
	By Proposition \ref{simone}.1, it follows that $\vdash\mathbf{Sim}(\vv)\rightarrow\neg\varphi$. Using $\mathrm{Nec_\blacksquare}$ and $\mathrm{K_\blacksquare}$ together with the fact that $\vv\in\varrho(\ww)$, we get that 
	$$ \vdash \blacksquare\bigvee_{\vv\in\varrho(\ww) }\mathbf{Sim}(\vv)\rightarrow \blacksquare \neg\varphi.$$
	By hypothetical syllogism combining this with (\ref{hyposyl}) yields $\vdash \mathbf{Sim}(\ww)\rightarrow \blacksquare \neg\varphi$. By the assumption that  $\blacklozenge\varphi\in\ell(\ww)$ together with Proposition \ref{simone}.1, it follows that $\vdash \mathbf{Sim}(\ww)\rightarrow  \blacklozenge\varphi$. Hence, $\vdash\neg\mathbf{Sim}(\ww)$ in contradiction since $\ww\in|\cqm_\Sigma|$.
\end{proof}
Putting together the above results, we conclude that $\cqm_\Sigma$ is always a quasimodel.
\begin{corollary}\label{canonicalisquasi}
	Given $\Sigma\Subset\fulllan$, the canonical structure $\cqm_\Sigma$ is a quasimodel.
\end{corollary}

\begin{proof}
	Combining Lemma \ref{openserial} with Lemma \ref{omegasens}, it follows that $|\cqm_\Sigma|$ is open in $|\univ|$ and that $\mapsto_{\cqm_\Sigma}$ is serial and $\omega$-sensible. 
    It follows by definition that $\cqm_\Sigma$ is regular and therefore by Proposition \ref{regular} the canonical structure $\cqm_\Sigma$ is a quasimodel.
\end{proof}
\subsection{Completeness}
We now have all the tools needed to prove completeness for $\mathbf{DGL}$.
\begin{proof}[Proof of Theorem \ref{dglcomp}]
	\sloppy Recall that a logic $\Lambda$ is complete if and only if every $\Lambda$-consistent formula is satisfied on a $\Lambda$-model.
	
	 Let $\varphi\in\fulllan$ be a consistent formula, i.e.\ $ \not\vdash\neg\varphi $. 
	 Let $\Sigma=\sub_\pm(\varphi)$.
	 Since clearly $\varphi\in \Sigma$, it follows from Proposition \ref{simone}.4 that $$\vdash\varphi\rightarrow\bigvee\big\{\Sim(\ww): \ww\in |\univnp^0_\Sigma|\text{ and }\varphi\in\ell(\ww)\big\}.$$
Since $\varphi$ is consistent, the disjunction above is consistent.
Accordingly, there exists $\ww\in |\univnp^0_\Sigma|$ for which $\Sim(\ww)$ is consistent and so $\ww\in|\cqm_\Sigma|$.
By Corollary \ref{canonicalisquasi}, the canonical structure $\cqm_\Sigma$ is a quasimodel. 
Therefore, by Theorem \ref{quastolim} we obtain $\vec{\cqm}_\Sigma \models \varphi$ and so there exists a scattered dynamical model that satisfies $\varphi$.
\end{proof}

\section{Conclusion}

We have exhibited the first finitely axiomatisable dynamic topological logic in the original trimodal language.
The techniques employed here can be applied to related logics which may or may not be topologically inspired, including expanding products of modal logics.\footnote{See Gabelaia, Kurucz, Wolter and Zakharyaschev \cite{pml} for the connection between expanding products and dynamic topological logic.}
In particular, dynamic Grzegorczyk logic ($\mathbf{DGrz}$) could be treated in the same fashion, where $\lozenge$ is interpreted as closure rather than Cantor derivative. Note, however, that the Cantor derivative can define the topological closure, so completeness for $\mathbf{DGrz}$ should also follow from embedding it into $\mathbf{DGL}$ using proof-translation techniques.

In fact, tangle-free logics may be applicable to a wider class of topological spaces by modifying the underlying Boolean algebra.
Instead of considering the powerset of $X$, one may work on sub-algebras (i.e., regular open or closed sets~\cite{KP-HZ14}).
In this setting, the tangled operators could also be trivialised, eliminating the need for such operators without restricting the class of topological spaces at one's disposal.

Finally, there is the question of axiomatising the dynamic topological logic of Aleksandroff spaces.
Chopoghloo and Moniri~\cite{Chopo20} provided an infinitary proof system for this class, and the results of Fern\'andez-Duque~\cite{FernandezNonFin} apply in this setting as well and rule out a finite axiomatisation.
However, it is possible that a natural, finitary proof system can be found in this setting (albeit with infinitely many axioms).

\bibliographystyle{plain}
\bibliography{biblio}




\
\end{document}